\documentclass{amsart}
\usepackage{amsmath}
\usepackage{amssymb}
\usepackage{amsfonts}

\newtheorem{theorem}{Theorem}
\theoremstyle{plain}

\newtheorem{lemma}{Lemma}

\newtheorem{proposition}{Proposition}

\numberwithin{equation}{section}
\input{tcilatex}

\begin{document}
\title[Hermite-Hadamard type]{some new inequalities of Hermite-Hadamard's
type }
\author{Aziz Saglam}
\address{Department of Mathematics,Faculty of Science and Arts, Afyon
Kocatepe University, Afyon, Turkey}
\email{azizsaglam@aku.edu.tr}
\thanks{$^{\star }$corresponding author}
\author{Mehmet Zeki Sar\i kaya$^{\star }$}
\address{Department of Mathematics,Faculty of Science and Arts, D\"{u}zce
University, D\"{u}zce, Turkey}
\email{sarikayamz@gmail.com, sarikaya@aku.edu.tr}
\author{Huseyin Y\i ld\i r\i m}
\address{Department of Mathematics,Faculty of Science and Arts, Afyon
Kocatepe University, Afyon, Turkey}
\email{hyildir@aku.edu.tr}
\date{}

\begin{abstract}
In this paper, we establish several new inequalities for some differantiable
mappings that are connected with the celebrated Hermite-Hadamard integral
inequality. Some applications for special means of real numbers are also
provided.
\end{abstract}

\maketitle

\section{Introduction}

The following inequality is well known in the literature as the
Hermite-Hadamard integral inequality (see, \cite{PPT}):

\begin{equation}
f\left( \frac{a+b}{2}\right) \leq \frac{1}{b-a}\int_{a}^{b}f(x)dx\leq \frac{%
f(a)+f(b)}{2}  \label{H}
\end{equation}%
where $f:I\subset \mathbb{R}\rightarrow \mathbb{R}$ is a convex function on
the interval $I$ of real numbers and $a,b\in I$ with $a<b$.A function $%
f:[a,b]\subset \mathbb{R}\rightarrow \mathbb{R}$ is said to be convex if \
whenever $x,y\in \lbrack a,b]$ and $t\in \left[ 0,1\right] $, the following
inequality holds%
\begin{equation*}
f(tx+(1-t)y)\leq tf(x)+(1-t)f(y).
\end{equation*}%
This definition has its origins in Jensen's results from \cite{jensen} and
has opened up the most extended, useful and multi-disciplinary domain of
mathematics, namely, canvex analysis. Convex curvers and convex bodies have
appeared in mathematical literature since antiquity and there are many
important resuls related to them. We say that $f$ is concave if $(-f)$ is
convex.

A largely applied inequality for convex functions, due to its geometrical
significance, is Hadamard's inequality, (see \cite{SSDRPA},\cite{CEMPJP}, 
\cite{USK} and \cite{USKMEO}) which has generated a wide range of directions
for extension and a rich mathematical literature.

In \cite{USK} in order to prove some inequalities related to Hadamard's
inequality K\i rmac\i\ used the following lemma:

\begin{lemma}
\label{l1} Let $f:I^{\circ }\subset \mathbb{R}\rightarrow \mathbb{R}$, be a
differentiable mapping on $I^{\circ }$, $a,b\in I^{\circ }$ ($I^{\circ }$ is
the interior of $I$) with $a<b$. If \ $f^{\prime }\in L\left( \left[ a,b%
\right] \right) $, then we have%
\begin{equation*}
\begin{array}{l}
\dfrac{1}{b-a}\dint_{a}^{b}f(x)dx-f\left( \dfrac{a+b}{2}\right) \\ 
\\ 
\ \ \ \ \ =\left( b-a\right) \left[ \dint_{0}^{\frac{1}{2}}tf^{\prime
}(ta+(1-t)b)dt+\dint_{\frac{1}{2}}^{1}\left( t-1\right) f^{\prime
}(ta+(1-t)b)dt\right] .%
\end{array}%
\end{equation*}
\end{lemma}

Also, in \cite{USKMEO}, K\i rmac\i\ and \"{O}zdemir obtained the following
inequality for differeftiable mappings which are connected with
Hermite-Hadamard's inequality:

\begin{theorem}
\label{l2} Let $f:I^{\circ }\subset \mathbb{R}\rightarrow \mathbb{R}$, be a
differentiable mapping on $I^{\circ }$, $a,b\in I^{\circ }$ with $a<b$ and $%
p>1$. If the mapping $\left\vert f^{\prime }\right\vert ^{p}$ is convex on $%
\left[ a,b\right] $, then%
\begin{equation*}
\left\vert \frac{1}{b-a}\int_{a}^{b}f(x)dx-f\left( \frac{a+b}{2}\right)
\right\vert \leq \frac{\left( 3^{1-\frac{1}{q}}\right) }{8}\left( b-a\right)
\left( \left\vert f^{\prime }(a)\right\vert +\left\vert f^{\prime
}(b)\right\vert \right) .
\end{equation*}
\end{theorem}

In this article, using functions whose derivatives absolute values are
convex, we obtained new inequalities releted to the left side of
Hermite-Hadamard inequality. Finally, we gave some applications for special
means of real numbers.

\section{Main Results}

We start with the following lemma:

\begin{lemma}
\label{z} Let $f:I^{\circ }\subset \mathbb{R}\rightarrow \mathbb{R}$ be a
differentiable mapping on $I^{\circ }$, $a,b\in I^{\circ }$ with $a<b$. If\ $%
f^{\prime }\in L[a,b]$, then the following equality holds:%
\begin{equation}
\begin{array}{l}
f(\dfrac{a+b}{2})-\dfrac{1}{b-a}\dint_{a}^{b}f(x)dx \\ 
\\ 
\ \ \ \ \ \ \ \ \ \ =\dfrac{b-a}{2}\dint_{0}^{1}\dint_{0}^{1}\left(
f^{\prime }(ta+(1-t)b)-f^{\prime }(sa+(1-s)b)\right) \left( m\left( s\right)
-m\left( t\right) \right) dtds.%
\end{array}
\label{10}
\end{equation}%
with%
\begin{equation*}
m(.):=\left\{ 
\begin{array}{ll}
t & ,t\in \lbrack 0,\frac{1}{2}] \\ 
&  \\ 
t-1 & ,t\in (\frac{1}{2},1]%
\end{array}%
\right.
\end{equation*}
\end{lemma}

\begin{proof}
By definitions of $m(.),$ it follows that%
\begin{equation*}
\begin{array}{l}
\dint_{0}^{1}\dint_{0}^{1}\left( f^{\prime }(ta+(1-t)b)-f^{\prime
}(sa+(1-s)b)\right) \left( m\left( t\right) -m\left( s\right) \right) dtds
\\ 
\\ 
\ \ \ =\dint_{0}^{1}\left\{ \dint_{0}^{1}f^{\prime }(ta+(1-t)b)\left(
m\left( t\right) -m\left( s\right) \right) dt-\dint_{0}^{1}f^{\prime
}(sa+(1-s)b)\left( m\left( t\right) -m\left( s\right) \right) dt\right\} ds
\\ 
\\ 
\ \ \ =\dint_{0}^{1}\left\{ \dint_{0}^{1/2}f^{\prime }(ta+(1-t)b)\left(
t-m\left( s\right) \right) dt+\dint_{1/2}^{1}f^{\prime }(ta+(1-t)b)\left(
t-1-m\left( s\right) \right) dt\right\} ds \\ 
\\ 
\ \ \ -\dint_{0}^{1}\left\{ \dint_{0}^{1/2}f^{\prime }(sa+(1-s)b)\left(
t-m\left( s\right) \right) dtdt+\dint_{1/2}^{1}f^{\prime }(sa+(1-s)b)\left(
t-1-m\left( s\right) \right) dt\right\} ds \\ 
\\ 
\ \ \ =\dint_{0}^{1/2}\left\{ \dint_{0}^{1/2}f^{\prime }(ta+(1-t)b)\left(
t-s\right) dt\right\} ds+\dint_{1/2}^{1}\left\{ \dint_{0}^{1/2}f^{\prime
}(ta+(1-t)b)\left( t-s+1\right) dt\right\} ds \\ 
\\ 
\ \ \ +\dint_{0}^{1/2}\left\{ \dint_{1/2}^{1}f^{\prime }(ta+(1-t)b)\left(
t-s-1\right) dt\right\} ds+\dint_{1/2}^{1}\left\{ \dint_{1/2}^{1}f^{\prime
}(ta+(1-t)b)\left( t-s\right) dt\right\} ds%
\end{array}%
\end{equation*}%
\begin{equation*}
\begin{array}{l}
\ \ \ -\dint_{0}^{1/2}\left\{ \dint_{0}^{1/2}f^{\prime }(sa+(1-s)b)\left(
t-s\right) dt\right\} ds-\dint_{1/2}^{1}\left\{ \dint_{0}^{1/2}f^{\prime
}(sa+(1-s)b)\left( t-s+1\right) dt\right\} ds \\ 
\\ 
\ \ \ -\dint_{0}^{1/2}\left\{ \dint_{1/2}^{1}f^{\prime }(sa+(1-s)b)\left(
t-s-1\right) dt\right\} ds-\dint_{1/2}^{1}\left\{ \dint_{1/2}^{1}f^{\prime
}(sa+(1-s)b)\left( t-s\right) dt\right\} ds \\ 
\\ 
\ \ \ =I_{1}+I_{2}+I_{3}+I_{4}-I_{5}-I_{6}-I_{7}-I_{8}.%
\end{array}%
\end{equation*}%
Thus by integration by parts, we can state:%
\begin{equation}
\begin{array}{l}
I_{1}=\dint_{0}^{1/2}\left\{ \dint_{0}^{1/2}f^{\prime }(ta+(1-t)b)\left(
t-s\right) dt\right\} ds \\ 
\\ 
\text{ \ \ }=\dint_{0}^{1/2}\left\{ \left( t-s\right) \dfrac{f(ta+(1-t)b)}{%
a-b}\underset{0}{\overset{1/2}{\mid }}-\dfrac{1}{\left( a-b\right) }%
\dint_{0}^{1/2}f(ta+(1-t)b)dt\right\} ds \\ 
\\ 
\text{ \ \ }=\dint_{0}^{1/2}\left\{ \left( \frac{1}{2}-s\right) \dfrac{f(%
\frac{a+b}{2})}{a-b}+s\dfrac{f(b)}{a-b}\right\} ds-\dfrac{1}{2\left(
a-b\right) }\dint_{0}^{1/2}f(ta+(1-t)b)dt \\ 
\\ 
\text{ \ \ }=\left( \left( \dfrac{s}{2}-\dfrac{s^{2}}{2}\right) \dfrac{f(%
\frac{a+b}{2})}{a-b}+\dfrac{s^{2}}{2}\dfrac{f(b)}{a-b}\right) \underset{0}{%
\overset{1/2}{\mid }}-\dfrac{1}{2\left( a-b\right) }%
\dint_{0}^{1/2}f(ta+(1-t)b)dt \\ 
\\ 
\text{ \ \ }=\dfrac{1}{8}\dfrac{f(\frac{a+b}{2})}{a-b}+\dfrac{1}{8}\dfrac{%
f(b)}{a-b}-\dfrac{1}{2\left( a-b\right) }\dint_{0}^{1/2}f(ta+(1-t)b)dt,%
\end{array}
\label{1}
\end{equation}%
\begin{equation}
\begin{array}{l}
I_{2}=\dint_{1/2}^{1}\left\{ \dint_{0}^{1/2}f^{\prime }(ta+(1-t)b)\left(
t-s+1\right) dt\right\} ds \\ 
\\ 
\text{ \ \ }=\dint_{1/2}^{1}\left\{ \left( t-s+1\right) \dfrac{f(ta+(1-t)b)}{%
a-b}\underset{0}{\overset{1/2}{\mid }}-\dfrac{1}{\left( a-b\right) }%
\dint_{0}^{1/2}f(ta+(1-t)b)dt\right\} ds \\ 
\\ 
\text{ \ \ }=\dint_{1/2}^{1}\left\{ \left( \dfrac{3}{2}-s\right) \dfrac{f(%
\frac{a+b}{2})}{a-b}+\left( s-1\right) \dfrac{f(b)}{a-b}\right\} ds-\dfrac{1%
}{2\left( a-b\right) }\dint_{0}^{1/2}f(ta+(1-t)b)dt \\ 
\\ 
\text{ \ \ }=\left( \left( \dfrac{3s}{2}-\dfrac{s^{2}}{2}\right) \dfrac{f(%
\frac{a+b}{2})}{a-b}+\left( \dfrac{s^{2}}{2}-s\right) \dfrac{f(b)}{a-b}%
\right) \underset{1/2}{\overset{1}{\mid }}-\dfrac{1}{2\left( a-b\right) }%
\dint_{0}^{1/2}f(ta+(1-t)b)dt \\ 
\\ 
\text{ \ \ }=\dfrac{3}{8}\dfrac{f(\frac{a+b}{2})}{a-b}-\dfrac{1}{8}\dfrac{%
f(b)}{a-b}-\dfrac{1}{2\left( a-b\right) }\dint_{0}^{1/2}f(ta+(1-t)b)dt,%
\end{array}
\label{2}
\end{equation}%
\begin{equation}
\begin{array}{l}
I_{3}=\dint_{0}^{1/2}\left\{ \dint_{1/2}^{1}f^{\prime }(ta+(1-t)b)\left(
t-s-1\right) dt\right\} ds \\ 
\text{ \ \ }=\dint_{0}^{1/2}\left\{ \left( t-s-1\right) \dfrac{f(ta+(1-t)b)}{%
a-b}\underset{1/2}{\overset{1}{\mid }}-\dfrac{1}{\left( a-b\right) }%
\dint_{1/2}^{1}f(ta+(1-t)b)dt\right\} ds \\ 
\\ 
\text{ \ \ }=\dint_{0}^{1/2}\left\{ \left( s+\frac{1}{2}\right) \dfrac{f(%
\frac{a+b}{2})}{a-b}-s\dfrac{f(a)}{a-b}\right\} ds-\dfrac{1}{2\left(
a-b\right) }\dint_{1/2}^{1}f(ta+(1-t)b)dt \\ 
\\ 
\text{ \ \ }=\left( \left( \dfrac{s^{2}}{2}+\dfrac{s}{2}\right) \dfrac{f(%
\frac{a+b}{2})}{a-b}-\dfrac{s^{2}}{2}\dfrac{f(a)}{a-b}\right) \underset{0}{%
\overset{1/2}{\mid }}-\dfrac{1}{2\left( a-b\right) }%
\dint_{1/2}^{1}f(ta+(1-t)b)dt \\ 
\\ 
\text{ \ \ }=\dfrac{3}{8}\dfrac{f(\frac{a+b}{2})}{a-b}-\dfrac{1}{8}\dfrac{%
f(a)}{a-b}-\dfrac{1}{2\left( a-b\right) }\dint_{1/2}^{1}f(ta+(1-t)b)dt,%
\end{array}
\label{3}
\end{equation}%
\begin{equation}
\begin{array}{l}
I_{4}=\dint_{1/2}^{1}\left\{ \dint_{1/2}^{1}f^{\prime }(ta+(1-t)b)\left(
t-s\right) dt\right\} ds \\ 
\text{ \ \ }=\dint_{1/2}^{1}\left\{ \left( t-s\right) \dfrac{f(ta+(1-t)b)}{%
a-b}\underset{1/2}{\overset{1}{\mid }}-\dfrac{1}{\left( a-b\right) }%
\dint_{1/2}^{1}f(ta+(1-t)b)dt\right\} ds \\ 
\\ 
\text{ \ \ }=\dint_{1/2}^{1}\left\{ \left( s-\frac{1}{2}\right) \dfrac{f(%
\frac{a+b}{2})}{a-b}+\left( 1-s\right) \dfrac{f(a)}{a-b}\right\} ds-\dfrac{1%
}{2\left( a-b\right) }\dint_{1/2}^{1}f(ta+(1-t)b)dt \\ 
\\ 
\text{ \ \ }=\left( \left( \dfrac{s^{2}}{2}-\dfrac{s}{2}\right) \dfrac{f(%
\frac{a+b}{2})}{a-b}+\left( s-\dfrac{s^{2}}{2}\right) \dfrac{f(a)}{a-b}%
\right) \underset{1/2}{\overset{1}{\mid }}-\dfrac{1}{2\left( a-b\right) }%
\dint_{1/2}^{1}f(ta+(1-t)b)dt \\ 
\\ 
\text{ \ \ }=\dfrac{1}{8}\dfrac{f(\frac{a+b}{2})}{a-b}+\dfrac{1}{8}\dfrac{%
f(a)}{a-b}-\dfrac{1}{2\left( a-b\right) }\dint_{1/2}^{1}f(ta+(1-t)b)dt,%
\end{array}
\label{4}
\end{equation}%
\begin{equation}
\begin{array}{l}
I_{5}=\dint_{0}^{1/2}\left\{ \dint_{0}^{1/2}f^{\prime }(sa+(1-s)b)\left(
t-s\right) dt\right\} ds \\ 
\text{ \ \ }=\dint_{0}^{1/2}\left\{ \left( \frac{t^{2}}{2}-st\right)
f^{\prime }(sa+(1-s)b)\underset{0}{\overset{1/2}{\mid }}\right\} ds \\ 
\\ 
\text{ \ \ }=\dint_{0}^{1/2}\left( \dfrac{1}{8}-\dfrac{s}{2}\right)
f^{\prime }(sa+(1-s)b)ds \\ 
\\ 
\text{ \ \ }=\left( \dfrac{1}{8}-\dfrac{s}{2}\right) \dfrac{f(sa+(1-s)b))}{%
a-b}\underset{0}{\overset{1/2}{\mid }}+\dfrac{1}{2\left( a-b\right) }%
\dint_{0}^{1/2}f(sa+(1-s)b)ds \\ 
\\ 
\text{ \ \ }=-\dfrac{1}{8}\dfrac{f(\frac{a+b}{2})}{a-b}-\dfrac{1}{8}\dfrac{%
f(b)}{a-b}+\dfrac{1}{2\left( a-b\right) }\dint_{0}^{1/2}f(sa+(1-s)b)ds,%
\end{array}
\label{5}
\end{equation}%
\begin{equation}
\begin{array}{l}
I_{6}=\dint_{1/2}^{1}\left\{ \dint_{0}^{1/2}f^{\prime }(sa+(1-s)b)\left(
t-s+1\right) dt\right\} ds \\ 
\text{ \ \ }=\dint_{1/2}^{1}\left\{ \left( \dfrac{t^{2}}{2}-st+t\right)
f^{\prime }(sa+(1-s)b)\underset{0}{\overset{1/2}{\mid }}\right\} ds \\ 
\\ 
\text{ \ \ }=\dint_{1/2}^{1}\left( \dfrac{5}{8}-\dfrac{s}{2}\right)
f^{\prime }(sa+(1-s)b)ds \\ 
\\ 
\text{ \ \ }=\left( \dfrac{5}{8}-\dfrac{s}{2}\right) \dfrac{f(sa+(1-s)b))}{%
a-b}\underset{1/2}{\overset{1}{\mid }}+\dfrac{1}{2\left( a-b\right) }%
\dint_{1/2}^{1}f(sa+(1-s)b)ds \\ 
\\ 
\text{ \ \ }=-\dfrac{3}{8}\dfrac{f(\frac{a+b}{2})}{a-b}+\dfrac{1}{8}\dfrac{%
f(a)}{a-b}+\dfrac{1}{2\left( a-b\right) }\dint_{1/2}^{1}f(sa+(1-s)b)ds,%
\end{array}
\label{6}
\end{equation}%
\begin{equation}
\begin{array}{l}
I_{7}=\dint_{0}^{1/2}\left\{ \dint_{1/2}^{1}f^{\prime }(sa+(1-s)b)\left(
t-s-1\right) dt\right\} ds \\ 
\text{ \ \ }=\dint_{0}^{1/2}\left\{ \left( \dfrac{t^{2}}{2}-st-t\right)
f^{\prime }(sa+(1-s)b)\underset{1/2}{\overset{1}{\mid }}\right\} ds \\ 
\\ 
\text{ \ \ }=\dint_{0}^{1/2}\left( -\dfrac{1}{8}-\dfrac{s}{2}\right)
f^{\prime }(sa+(1-s)b)ds \\ 
\\ 
\text{ \ \ }=\left( -\dfrac{1}{8}-\dfrac{s}{2}\right) \dfrac{f(sa+(1-s)b))}{%
a-b}\underset{0}{\overset{1/2}{\mid }}+\dfrac{1}{2\left( a-b\right) }%
\dint_{0}^{1/2}f(sa+(1-s)b)ds \\ 
\\ 
\text{ \ \ }=-\dfrac{3}{8}\dfrac{f(\frac{a+b}{2})}{a-b}+\dfrac{1}{8}\dfrac{%
f(b)}{a-b}+\dfrac{1}{2\left( a-b\right) }\dint_{0}^{1/2}f(sa+(1-s)b)ds,%
\end{array}
\label{7}
\end{equation}%
\begin{equation}
\begin{array}{l}
I_{8}=\dint_{1/2}^{1}\left\{ \dint_{1/2}^{1}f^{\prime }(sa+(1-s)b)\left(
t-s\right) dt\right\} ds \\ 
\text{ \ \ }=\dint_{1/2}^{1}\left\{ \left( \dfrac{t^{2}}{2}-st\right)
f^{\prime }(sa+(1-s)b)\underset{1/2}{\overset{1}{\mid }}\right\} ds \\ 
\\ 
\text{ \ \ }=\dint_{1/2}^{1}\left( \dfrac{3}{8}-\dfrac{s}{2}\right)
f^{\prime }(sa+(1-s)b)ds \\ 
\\ 
\text{ \ \ }=\left( \dfrac{3}{8}-\dfrac{s}{2}\right) \dfrac{f(sa+(1-s)b))}{%
a-b}\underset{1/2}{\overset{1}{\mid }}+\dfrac{1}{2\left( a-b\right) }%
\dint_{1/2}^{1}f(sa+(1-s)b)ds \\ 
\\ 
\text{ \ \ }=-\dfrac{1}{8}\dfrac{f(\frac{a+b}{2})}{a-b}-\dfrac{1}{8}\dfrac{%
f(a)}{a-b}+\dfrac{1}{2\left( a-b\right) }\dint_{1/2}^{1}f(ta+(1-t)b)dt.%
\end{array}
\label{8}
\end{equation}%
Adding (\ref{1})-(\ref{8}) and rewritting, we easily deduce: 
\begin{equation*}
\begin{array}{l}
\dint_{0}^{1}\dint_{0}^{1}\left( f^{\prime }(ta+(1-t)b)-f^{\prime
}(sa+(1-s)b)\right) \left( m\left( t\right) -m\left( s\right) \right) dtds
\\ 
\\ 
\text{ \ \ }=I_{1}+I_{2}+I_{3}+I_{4}-I_{5}-I_{6}-I_{7}-I_{8} \\ 
\\ 
\text{ \ \ }=2\dfrac{f(\frac{a+b}{2})}{a-b}-\dfrac{2}{\left( a-b\right) }%
\dint_{0}^{1}f(ta+(1-t)b)dt.%
\end{array}%
\end{equation*}%
Using the change of the variable $x=ta+(1-t)b$ for $t\in \lbrack 0,1]$, and
multiplying the both sides by $\left( a-b\right) /2,$ we obtain (\ref{10}),
which completes the proof.
\end{proof}

\begin{theorem}
\label{z2} Let $f:I^{\circ }\subset \mathbb{R}\rightarrow \mathbb{R}$ be a
differentiable mapping on $I^{\circ }$, $a,b\in I^{\circ }$ with $a<b$. If \ 
$\left\vert f^{\prime }\right\vert ^{2}$ is convex on $\left[ a,b\right] ,$
then the following inequality holds:%
\begin{equation}
\left\vert f(\frac{a+b}{2})-\dfrac{1}{b-a}\dint_{a}^{b}f(x)dx\right\vert
\leq \frac{b-a}{\sqrt{6}}\left( \frac{\left\vert f^{\prime }(a)\right\vert
^{2}+\left\vert f^{\prime }(b)\right\vert ^{2}}{2}\right) ^{\frac{1}{2}}.
\label{9}
\end{equation}
\end{theorem}

\begin{proof}
From Lemma \ref{z}, using the Cauchy-Schwartz for double integrals, we get%
\begin{equation}
\begin{array}{l}
\left\vert f(\dfrac{a+b}{2})-\dfrac{1}{b-a}\dint_{a}^{b}f(x)dx\right\vert \\ 
\\ 
\ \ \ \ \ =\dfrac{b-a}{2}\left\vert \dint_{0}^{1}\dint_{0}^{1}\left(
f^{\prime }(ta+(1-t)b)-f^{\prime }(sa+(1-s)b)\right) \left( m\left( s\right)
-m\left( t\right) \right) dtds\right\vert \\ 
\\ 
\ \ \ \ \ \leq \dfrac{b-a}{2}\left[ \dint_{0}^{1}\dint_{0}^{1}\left\vert
f^{\prime }(ta+(1-t)b)-f^{\prime }(sa+(1-s)b)\right\vert \left\vert m\left(
t\right) -m\left( s\right) \right\vert dtds\right] \\ 
\\ 
\ \ \ \ \ \leq \dfrac{b-a}{2}\left[ \dint_{0}^{1}\dint_{0}^{1}\left\vert
f^{\prime }(ta+(1-t)b)\right\vert \left\vert m\left( t\right) -m\left(
s\right) \right\vert dtds+\dint_{0}^{1}\dint_{0}^{1}\left\vert f^{\prime
}(sa+(1-s)b)\right\vert \left\vert m\left( t\right) -m\left( s\right)
\right\vert dtds\right] \\ 
\\ 
\ \ \ \ \ =\left( b-a\right) \dint_{0}^{1}\dint_{0}^{1}\left\vert f^{\prime
}(ta+(1-t)b)\right\vert \left\vert m\left( t\right) -m\left( s\right)
\right\vert dtds \\ 
\\ 
\ \ \ \ \ \leq \left( b-a\right) \left[ \left(
\dint_{0}^{1}\dint_{0}^{1}\left( m\left( t\right) -m\left( s\right) \right)
^{2}dtds\right) ^{\frac{1}{2}}\left( \dint_{0}^{1}\dint_{0}^{1}\left\vert
f^{\prime }(ta+(1-t)b)\right\vert ^{2}dtds\right) ^{\frac{1}{2}}\right]%
\end{array}
\label{11}
\end{equation}%
By definitions of $m(t)$ and $m(s)$ and by simple computation, we get%
\begin{equation}
\begin{array}{l}
\dint_{0}^{1}\dint_{0}^{1}\left( m\left( t\right) -m\left( s\right) \right)
^{2}dtds \\ 
\\ 
\text{ \ \ }=\dint_{0}^{1}\left\{ \dint_{0}^{1/2}\left( t-m\left( s\right)
\right) ^{2}dt+\dint_{1/2}^{1}\left( t-1-m\left( s\right) \right)
^{2}dt\right\} ds \\ 
\\ 
\text{ \ \ }=\dint_{0}^{1/2}\left\{ \dint_{0}^{1/2}\left( t-s\right)
^{2}dt\right\} ds+\dint_{1/2}^{1}\left\{ \dint_{0}^{1/2}\left( t-s+1\right)
^{2}dt\right\} ds \\ 
\\ 
\text{ \ \ }+\dint_{0}^{1/2}\left\{ \dint_{1/2}^{1}\left( t-s-1\right)
^{2}dt\right\} ds+\dint_{1/2}^{1}\left\{ \dint_{1/2}^{1}\left( t-s\right)
^{2}dt\right\} ds \\ 
\\ 
\text{ \ \ }=\dint_{0}^{1/2}\left\{ \dfrac{\left( t-s\right) ^{3}}{3}%
\underset{0}{\overset{1/2}{\mid }}\right\} ds+\dint_{1/2}^{1}\left\{ \dfrac{%
\left( t-s+1\right) ^{3}}{3}\underset{0}{\overset{1/2}{\mid }}\right\} ds \\ 
\\ 
\text{ \ \ }+\dint_{0}^{1/2}\left\{ \dfrac{\left( t-s-1\right) ^{3}}{3}%
\underset{1/2}{\overset{1}{\mid }}\right\} ds+\dint_{1/2}^{1}\left\{ \dfrac{%
\left( t-s\right) ^{3}}{3}\underset{1/2}{\overset{1}{\mid }}\right\} ds \\ 
\\ 
\text{ \ \ }=\dint_{0}^{1/2}\left\{ \dfrac{\left( 1-2s\right) ^{3}}{24}+%
\dfrac{s^{3}}{3}\right\} ds+\dint_{1/2}^{1}\left\{ \dfrac{\left( 3-2s\right)
^{3}}{24}+\dfrac{\left( s-1\right) ^{3}}{3}\right\} ds \\ 
\\ 
\text{ \ \ }+\dint_{0}^{1/2}\left\{ \dfrac{\left( 2s+1\right) ^{3}}{24}-%
\dfrac{s^{3}}{3}\right\} ds+\dint_{1/2}^{1}\left\{ \dfrac{\left( 2s-1\right)
^{3}}{3}+\dfrac{\left( 1-s\right) ^{3}}{3}\right\} ds \\ 
\\ 
\ \ \ =\dfrac{1}{6}%
\end{array}
\label{12}
\end{equation}%
and since $\left\vert f^{\prime }\right\vert ^{2}$ is convex on $\left[ a,b%
\right] ,$ we know that for $t\in \lbrack 0,1]$%
\begin{equation*}
\left\vert f^{\prime }(ta+(1-t)b)\right\vert ^{2}\leq t\left\vert f^{\prime
}(a)\right\vert ^{2}+(1-t)\left\vert f^{\prime }(b)\right\vert ^{2},
\end{equation*}%
hence%
\begin{equation}
\begin{array}{ll}
\left( \dint_{0}^{1}\dint_{0}^{1}\left\vert f^{\prime
}(ta+(1-t)b)\right\vert ^{2}dtds\right) ^{\frac{1}{2}} & \leq \left(
\dint_{0}^{1}\dint_{0}^{1}\left( t\left\vert f^{\prime }(a)\right\vert
^{2}+(1-t)\left\vert f^{\prime }(b)\right\vert ^{2}\right) dtds\right) ^{%
\frac{1}{2}} \\ 
&  \\ 
& =\left( \dfrac{\left\vert f^{\prime }(a)\right\vert ^{2}+\left\vert
f^{\prime }(b)\right\vert ^{2}}{2}\right) ^{\frac{1}{2}}.%
\end{array}
\label{13}
\end{equation}%
Therefore, using (\ref{12}) and (\ref{13}) in (\ref{11}), we obtain (\ref{9}%
).
\end{proof}

\begin{theorem}
\label{z3} Let $f:I^{\circ }\subset \mathbb{R}\rightarrow \mathbb{R}$ be a
differentiable mapping on $I^{\circ }$, $a,b\in I^{\circ }$ with $a<b$. If $%
\left\vert f^{\prime }\right\vert ^{q}$ is convex on $\left[ a,b\right] ,\
q>1,$ then the following inequality holds:%
\begin{equation}
\left\vert f(\frac{a+b}{2})-\dfrac{1}{b-a}\dint_{a}^{b}f(x)dx\right\vert
\leq \left( b-a\right) \left( \frac{2}{\left( p+1\right) \left( p+2\right) }%
\right) ^{\frac{1}{p}}\left( \dfrac{\left\vert f^{\prime }(a)\right\vert
^{q}+\left\vert f^{\prime }(b)\right\vert ^{q}}{2}\right) ^{\frac{1}{q}}
\label{14}
\end{equation}%
where $\frac{1}{p}+\frac{1}{q}=1.$
\end{theorem}

\begin{proof}
From Lemma \ref{z}, we observe that%
\begin{equation}
\begin{array}{l}
\left\vert f(\dfrac{a+b}{2})-\dfrac{1}{b-a}\dint_{a}^{b}f(x)dx\right\vert \\ 
\\ 
\ \ \ \ \ \ \ \ \ \ =\dfrac{b-a}{2}\left\vert
\dint_{0}^{1}\dint_{0}^{1}\left( f^{\prime }(ta+(1-t)b)-f^{\prime
}(sa+(1-s)b)\right) \left( m\left( s\right) -m\left( t\right) \right)
dtds\right\vert \\ 
\\ 
\ \ \ \ \ \ \ \ \ \ \leq \dfrac{b-a}{2}\left[ \dint_{0}^{1}\dint_{0}^{1}%
\left\vert f^{\prime }(ta+(1-t)b)-f^{\prime }(sa+(1-s)b)\right\vert
\left\vert m\left( t\right) -m\left( s\right) \right\vert dtds\right] \\ 
\\ 
\ \ \ \ \ \ \ \ \ \ \leq \dfrac{b-a}{2}\left[ \dint_{0}^{1}\dint_{0}^{1}%
\left\vert f^{\prime }(ta+(1-t)b)\right\vert \left\vert m\left( t\right)
-m\left( s\right) \right\vert dtds\right. \\ 
\\ 
\ \ \ \ \ \ \ \ \ \ \left. +\dint_{0}^{1}\dint_{0}^{1}\left\vert f^{\prime
}(sa+(1-s)b)\right\vert \left\vert m\left( t\right) -m\left( s\right)
\right\vert dtds\right] \\ 
\\ 
\ \ \ \ \ \ \ \ \ \ \leq \left( b-a\right)
\dint_{0}^{1}\dint_{0}^{1}\left\vert f^{\prime }(ta+(1-t)b)\right\vert
\left\vert m\left( t\right) -m\left( s\right) \right\vert dtds \\ 
\\ 
\ \ \ \ \ \ \ \ \ \ \leq \left( b-a\right) \left[ \left(
\dint_{0}^{1}\dint_{0}^{1}\left\vert m\left( t\right) -m\left( s\right)
\right\vert ^{p}dtds\right) ^{\frac{1}{p}}\left(
\dint_{0}^{1}\dint_{0}^{1}\left\vert f^{\prime }(ta+(1-t)b)\right\vert
^{q}dtds\right) ^{\frac{1}{q}}\right] .%
\end{array}
\label{15}
\end{equation}%
By definitions of $m(t)$ and $m(s)$, we get%
\begin{equation*}
\begin{array}{l}
\dint_{0}^{1}\dint_{0}^{1}\left\vert m\left( t\right) -m\left( s\right)
\right\vert ^{p}dtds \\ 
\\ 
\text{ \ \ }=\dint_{0}^{1}\left\{ \dint_{0}^{1/2}\left\vert t-m\left(
s\right) \right\vert ^{p}dt+\dint_{1/2}^{1}\left\vert t-1-m\left( s\right)
\right\vert ^{p}dt\right\} ds \\ 
\\ 
\text{ \ \ }=\dint_{0}^{1/2}\dint_{0}^{1/2}\left\vert t-s\right\vert
^{p}dtds+\dint_{1/2}^{1}\dint_{0}^{1/2}\left\vert t-s+1\right\vert ^{p}dtds
\\ 
\\ 
\text{ \ \ }+\dint_{0}^{1/2}\dint_{1/2}^{1}\left\vert t-s-1\right\vert
^{p}dtds+\dint_{1/2}^{1}\dint_{1/2}^{1}\left\vert t-s\right\vert ^{p}dtds \\ 
\\ 
\text{ \ \ }=J_{1}+J_{2}+J_{3}+J_{4}.%
\end{array}%
\end{equation*}%
Thus, by simple computation we obtain%
\begin{equation}
\begin{array}{l}
J_{1}=\dint_{0}^{1/2}\dint_{0}^{1/2}\left\vert t-s\right\vert
^{p}dtds=\dint_{0}^{1/2}\left\{ \dint_{0}^{s}\left( s-t\right)
^{p}dt+\dint_{s}^{1/2}\left( t-s\right) ^{p}dt\right\} ds \\ 
\\ 
\text{ \ \ \ }=\dfrac{1}{p+1}\dint_{0}^{1/2}\left\{ s^{p+1}+\left( \dfrac{1}{%
2}-s\right) ^{p+1}\right\} ds=\dfrac{2}{2^{p+1}\left( p+1\right) \left(
p+2\right) },%
\end{array}
\label{17}
\end{equation}%
\begin{equation}
\begin{array}{l}
J_{2}=\dint_{1/2}^{1}\dint_{0}^{1/2}\left\vert t-s+1\right\vert
^{p}dtds=\dint_{1/2}^{1}\dint_{0}^{1/2}\left( t-s+1\right) ^{p}dtds,\text{ \
\ }\left( \text{ for }t-s+1\geq 0\right) \\ 
\\ 
\text{ \ \ \ }=\dfrac{1}{p+1}\dint_{1/2}^{1}\left\{ \left( \frac{3}{2}%
-s\right) ^{p+1}-\left( 1-s\right) ^{p+1}\right\} ds \\ 
\\ 
\text{ \ \ \ }=\dfrac{1}{\left( p+1\right) \left( p+2\right) }-\dfrac{1}{%
2^{p+1}\left( p+1\right) \left( p+2\right) },%
\end{array}
\label{18}
\end{equation}%
\begin{equation}
\begin{array}{l}
J_{3}=\dint_{0}^{1/2}\dint_{1/2}^{1}\left\vert t-s-1\right\vert
^{p}dtds=\dint_{0}^{1/2}\dint_{1/2}^{1}\left( -t+s+1\right) ^{p}dtds,\text{
\ \ }\left( \text{for }t-s-1\leq 0\right) \\ 
\\ 
\text{ \ \ \ }=\dfrac{1}{p+1}\dint_{0}^{1/2}\left\{ -s^{p+1}+\left( s+\frac{1%
}{2}\right) ^{p+1}\right\} ds \\ 
\\ 
\text{ \ \ \ }=\dfrac{1}{\left( p+1\right) \left( p+2\right) }-\dfrac{1}{%
2^{p+1}\left( p+1\right) \left( p+2\right) },%
\end{array}
\label{19}
\end{equation}%
\begin{equation}
\begin{array}{l}
J_{4}=\dint_{1/2}^{1}\dint_{1/2}^{1}\left\vert t-s\right\vert
^{p}dtds=\dint_{1/2}^{1}\left\{ \dint_{1/2}^{s}\left( s-t\right)
^{p}dt+\dint_{s}^{1}\left( t-s\right) ^{p}dt\right\} ds \\ 
\\ 
\text{ \ \ \ }=\dfrac{1}{p+1}\dint_{1/2}^{1}\left\{ \left( s-\frac{1}{2}%
\right) ^{p+1}+\left( 1-s\right) ^{p+1}\right\} ds \\ 
\\ 
\text{ \ \ \ }=\dfrac{1}{2^{p+1}\left( p+1\right) \left( p+2\right) }.%
\end{array}
\label{20}
\end{equation}%
Adding (\ref{17})-(\ref{20}), we have 
\begin{equation}
\left( \dint_{0}^{1}\dint_{0}^{1}\left\vert m\left( t\right) -m\left(
s\right) \right\vert ^{p}dtds\right) ^{\frac{1}{p}}=\left( \dfrac{2}{\left(
p+1\right) \left( p+2\right) }\right) ^{\frac{1}{p}}.  \label{21}
\end{equation}%
Since $\left\vert f^{\prime }\right\vert ^{q}$ is convex on $\left[ a,b%
\right] ,$ we know that for $t\in \lbrack 0,1]$%
\begin{equation*}
\left\vert f^{\prime }(ta+(1-t)b)\right\vert ^{q}\leq t\left\vert f^{\prime
}(a)\right\vert ^{q}+(1-t)\left\vert f^{\prime }(b)\right\vert ^{q},
\end{equation*}%
hence%
\begin{equation}
\begin{array}{ll}
\left( \dint_{0}^{1}\dint_{0}^{1}\left\vert f^{\prime
}(ta+(1-t)b)\right\vert ^{q}dtds\right) ^{\frac{1}{q}} & \leq \left(
\dint_{0}^{1}\dint_{0}^{1}\left( t\left\vert f^{\prime }(a)\right\vert
^{q}+(1-t)\left\vert f^{\prime }(b)\right\vert ^{q}\right) dtds\right) ^{%
\frac{1}{q}} \\ 
&  \\ 
& =\left( \dfrac{\left\vert f^{\prime }(a)\right\vert ^{q}+\left\vert
f^{\prime }(b)\right\vert ^{q}}{2}\right) ^{\frac{1}{q}}.%
\end{array}
\label{22}
\end{equation}%
Therefore, using (\ref{21}) and (\ref{22}) in (\ref{15}), we obtain (\ref{14}%
).
\end{proof}

\section{Applications to Some Special Means}

We now consider the applications of our Theorems to the following special
means:

(a) The arithmetic mean: $A=A(a,b):=\dfrac{a+b}{2},$ \ $a,b\geq 0,$

(b) The logarithmic mean: 
\begin{equation*}
L=L\left( a,b\right) :=\left\{ 
\begin{array}{ccc}
a & if & a=b \\ 
&  &  \\ 
\frac{b-a}{\ln b-\ln a} & if & a\neq b%
\end{array}%
\right. \text{, \ \ \ }a,b>0,
\end{equation*}

(c) The Identric mean:%
\begin{equation*}
I=I\left( a,b\right) :=\left\{ 
\begin{array}{ccc}
a & \text{if} & a=b \\ 
&  &  \\ 
\frac{1}{e}\left( \frac{b^{b}}{a^{a}}\right) ^{\frac{1}{b-a}}\text{ } & 
\text{if} & a\neq b%
\end{array}%
\right. \text{, \ \ \ }a,b>0,
\end{equation*}

(d) The $p-$logarithmic mean

\begin{equation*}
L_{p}=L_{p}(a,b):=\left\{ 
\begin{array}{ccc}
\left[ \frac{b^{p+1}-a^{p+1}}{\left( p+1\right) \left( b-a\right) }\right] ^{%
\frac{1}{p}} & \text{if} & a\neq b \\ 
&  &  \\ 
a & \text{if} & a=b%
\end{array}%
\right. \text{, \ \ \ }p\in \mathbb{R\diagdown }\left\{ -1,0\right\} ;\;a,b>0%
\text{.}
\end{equation*}%
The following proposition holds:

\begin{proposition}
\label{p.0} Let $a,b\in \mathbb{R}$, $0<a<b$ and $n\in 
\mathbb{Z}
$, $\left\vert n\right\vert \geq 1$. Then, we have%
\begin{equation*}
\left\vert A^{n}\left( a,b\right) -L_{n}^{n}\left( a,b\right) \right\vert
\leq \left\vert n\right\vert \frac{\left( b-a\right) }{\sqrt{6}}A\left(
a^{2(n-1)},b^{2(n-1)}\right) .
\end{equation*}
\end{proposition}

\begin{proof}
The proof is immediate from Theorem \ref{z2} applied for $f(x)=x^{n}$, $x\in 
\mathbb{R}$, $n\in \mathbb{%
\mathbb{Z}
}$ and $\left\vert n\right\vert \geq 1$.
\end{proof}

\begin{proposition}
\label{p.1} Let $a,b\in \mathbb{R}$, $0<a<b$ and $n\in 
\mathbb{Z}
$, $\left\vert n\right\vert \geq 1$. Then, for all $q>1,$we have%
\begin{equation*}
\left\vert A\left( a^{n},b^{n}\right) -L_{n}^{n}\left( a,b\right)
\right\vert \leq \left\vert n\right\vert (b-a)\left( \frac{2}{(p+1)(p+2)}%
\right) ^{\frac{1}{p}}\left[ A\left( \left\vert a\right\vert
^{q(n-1)},\left\vert b\right\vert ^{q(n-1)}\right) \right] ^{\frac{1}{q}}.
\end{equation*}

\begin{proof}
The assertion follows from Theorem \ref{z3} applied for $f(x)=x^{n}$, $x\in 
\mathbb{R}$, $n\in \mathbb{%
\mathbb{Z}
}$ and $\left\vert n\right\vert \geq 1$.
\end{proof}
\end{proposition}

\begin{proposition}
\label{p.2} Let $a,b\in \mathbb{R}$, $0<a<b$. Then, for all $q>1$, we have 
\begin{equation*}
\ln \left[ \frac{I\left( a,b\right) }{A\left( a,b\right) }\right] \leq \frac{%
\left( b-a\right) }{ab}\left( \frac{2}{(p+1)(p+2)}\right) ^{\frac{1}{p}}%
\left[ A\left( \left\vert b\right\vert ^{q},\left\vert a\right\vert
^{q}\right) \right] ^{\frac{1}{q}}.
\end{equation*}
\end{proposition}

\begin{proof}
The assertion follows from Theorem \ref{z3} applied to $f:(0,\infty
)\rightarrow (-\infty ,0),\;$ $f(x)=-\ln \left( x\right) $ and the details
are omitted.
\end{proof}

\begin{proposition}
\label{p.3} Let $a,b\in \mathbb{R}$, $0<a<b$. Then, for all $q\geq 1,$the
following inequality holds:%
\begin{equation*}
\left\vert A^{-1}\left( a,b\right) -L^{-1}\left( a,b\right) \right\vert \leq 
\frac{\left( b-a\right) }{\left( ab\right) ^{2}}\left( \frac{2}{(p+1)(p+2)}%
\right) ^{\frac{1}{p}}\left[ A\left( \left\vert a\right\vert
^{2q},\left\vert b\right\vert ^{2q}\right) \right] ^{\frac{1}{q}}.
\end{equation*}
\end{proposition}

\begin{proof}
The proof is obvious from Theorem \ref{z3} applied for $f(x)=\frac{1}{x},$ $%
x\in \left[ a,b\right] $.
\end{proof}

\end{document}